\documentclass[12pt]{article}

\pdfoutput=1
\usepackage[spanish,english]{babel}
\usepackage{amsmath}
\usepackage{amsfonts}
\usepackage{amsthm}
\usepackage{lineno}
\usepackage{epsfig, subfigure}
\usepackage{amscd,latexsym,amssymb}
\usepackage{graphicx}
\usepackage{url}
\usepackage{color}
\usepackage{enumerate}
\definecolor{gray}{gray}{0.5}

\newcommand{\bc}{\begin{center}}
\newcommand{\ec}{\end{center}}

\newcommand{\cc}{\textrm{cc}}
\newcommand{\cy}{\textrm{cy}}
\newcommand{\ex}{\textrm{ex}}

\newtheorem{theorem}{Theorem}
\newtheorem{proposition}{Proposition}
\newtheorem{corollary}{Corollary}
\newtheorem{lemma}{Lemma}

\newtheorem{defi}{Definition}

\newcommand\blfootnote[1]{%
 \begingroup
 \renewcommand\thefootnote{}\footnote{#1}%
 \addtocounter{footnote}{-1}%
 \endgroup
}
\begin{document}

%\begin{frontmatter}

\selectlanguage{english}

\title{
Locating domination in bipartite graphs and their complements
% \thanks{Partially supported by projects Gen. Cat. DGR 2014SGR46, MINECO MTM2015-63791-R and H2020-MSCA-RISE project 734922 - CONNECT.}
}

\author{
C. Hernando\thanks{Departament de Matem\`atiques. Universitat Polit\`ecnica de Catalunya, Barcelona, Spain, carmen.hernando@upc.edu.
Partially supported by projects Gen. Cat. DGR  2017SGR1336 and MTM2015-63791-R (MINECO/FEDER).}
\and M. Mora\thanks{Departament de Matem\`atiques. Universitat Polit\`ecnica de Catalunya, Barcelona, Spain, merce.mora@upc.edu. Partially supported by projects Gen. Cat. DGR 2017SGR1336, MTM2015-63791-R (MINECO/FEDER) and H2020-MSCA-RISE project 734922 - CONNECT.}
\and I. M. Pelayo\thanks{Departament de Matem\`atiques. Universitat Polit\`ecnica de Catalunya, Barcelona, Spain, ignacio.m.pelayo@upc.edu.
Partially supported by projects MINECO MTM2014-60127-P, ignacio.m.pelayo@upc.edu.}
}

\date{}

%%%%%%%%%%%%%%%%%%%%%%%%%%%%%%%%%%%%%%%%%%%%%%%%%%%%%%%%%%%%%%%%%%%%
%%%%%%%%%%%%%%%%%%%%%%%%%%%%%%%%%%%%%%%%%%%%%%%%%%%%%%%%%%%%%%%%%%%%
%%%%%%%%%%%%%%%%%%%%%%%%%%%%%%%%%%%%%%%%%%%%%%%%%%%%%%%%%%%%%%%%%%%%%%%%%%%%%%      
\maketitle

%\maketitle

\blfootnote{\begin{minipage}[l]{0.3\textwidth} \includegraphics[trim=10cm 6cm 10cm 5cm,clip,scale=0.15]{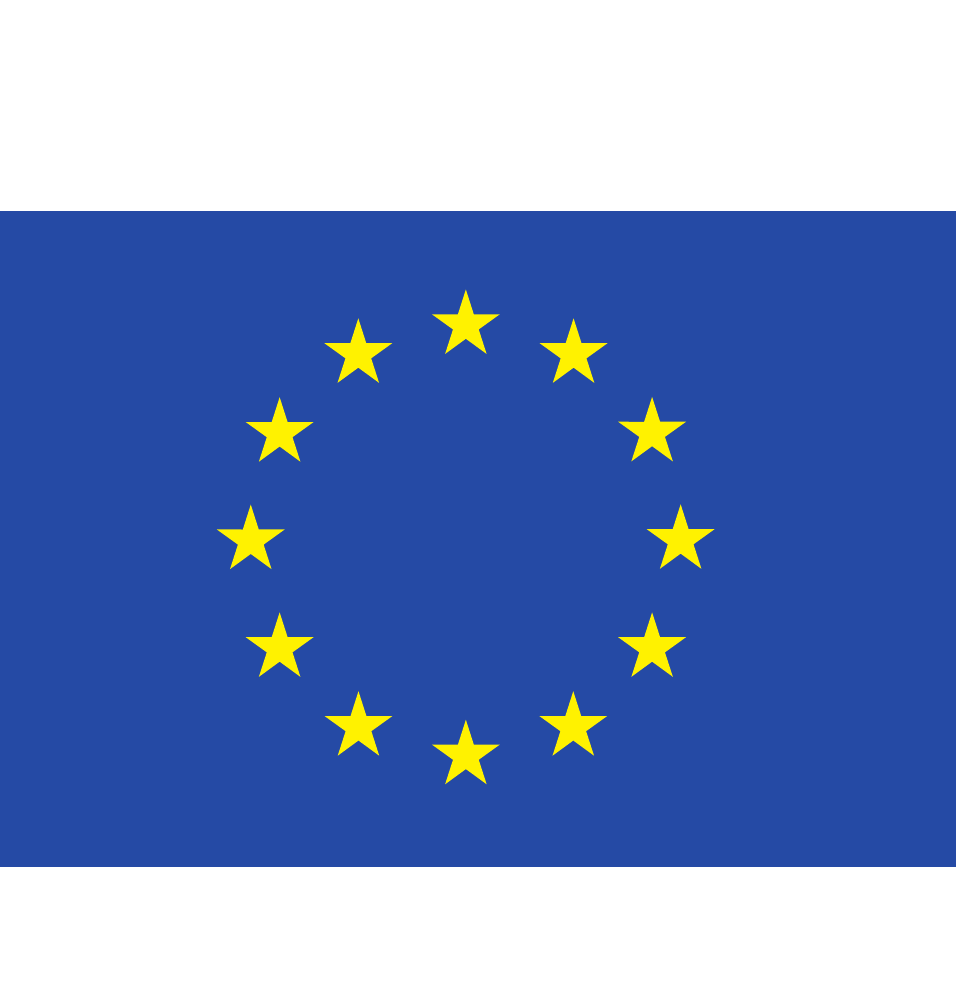} \end{minipage}  \hspace{-2cm} \begin{minipage}[l][1cm]{0.7\textwidth}
      This project has received funding from the European Union's Horizon 2020 research and innovation programme under the Marie Sk\l{}odowska-Curie grant agreement No 734922.
\end{minipage}}

%\linenumbers

\begin{abstract}
A set $S$  of vertices of a graph $G$ is \emph{distinguishing} if 
the sets of neighbors in $S$ for every pair of vertices not in $S$ are distinct.
A  \emph{locating-dominating set} of $G$ is a dominating distinguishing set.
The \emph{location-domination number} of $G$, $\lambda(G)$, is the minimum cardinality of a locating-dominating set.
In this work we study relationships between $\lambda({G})$ and $\lambda (\overline{G})$ for bipartite graphs. 
The main result is the characterization of all connected bipartite graphs $G$ satisfying $\lambda (\overline{G})=\lambda({G})+1$.
To this aim, we define an edge-labeled graph $G^S$ associated with a distinguishing set $S$ that turns out to be very helpful.
\end{abstract}



%%%%%%%%%%%%%%%%%%%%%%%%%%%%%%%%%%%%%%%%%%
%%%%%%%%%%%%%%%%%%%%%%%%%%%%%%%%%%%%%%%%%%
%%%%%%%%%%%%%%%%%%%%%%%%%%%%%%%%%%%%%%%%%%
%%%%%%%%%%%%%%%%%%%%%%%%%%%%%%%%%%%%%%%%%%
%%%%%%%%%%%%%%%%%%%%%%%%%%%%%%%%%%%%%%%%%%
%%%%%%%%%%%%%%%%%%%%%%%%%%%%%%%%%%%%%%%%%%
%%%%%%%%%%%%%%%%%%%%%%%%%%%%%%%%%%%%%%%%%%
%%%%%%%%%%%%%%%%%%%%%%%%%%%%%%%%%%%%%%%%%%

\newpage
%%%%%%%%%%%%%%%%%%%%%%%%%%%%%%%%%%%%%%%%%%%%%%%%%%%%%
%%%%%%%%%%%%%%%%%%%%%%%%%%%%%%%%%%%%%%%%%%%%%%%%%%%%%
%%%%%%%%%%%%%%%%%%%%%%%%%%%%%%%%%%%%%%%%%%%%%%%%%%%%%
%%%%%%%%%%%%%%%%%%%%%%%%%%%%%%%%%%%%%%%%%%%%%%%%%%%%%
%%%%%%%%%%%%%%%%%%%%%%%%%%%%%%%%%%%%%%%%%%%%%%%%%%%%%
%%%%%%%%%%%%%%%%%%%%%%%%%%%%%%%%%%%%%%%%%%%%%%%%%%%%%
\section{Introduction}

Let  $G=(V,E)$ be a simple, finite graph.
%The distance between two vertices $v$ and $w$ is denoted by $d_G(v,w)$.
The \emph{neighborhood} of a vertex $u\in V$ is $N_G(u)=\{ v : uv\in E\}$. 
We write $N(u)$ or $d(v,w)$ if the graph G is clear from the context.
For any $S\subseteq V$, $N(S)=\cup_{u\in S} N(u)$.
A set $S\subseteq V$ is \emph{dominating} if $V=S\cup N(S)$ (see \cite{hahesl}).
For further notation and terminology, we refer the reader to \cite{chlezh11}.

A set $S\subseteq V$  is \emph{distinguishing} if  $N(u)\cap S\not= N(v)\cap S$  
for every pair of different vertices $u,v \in V \setminus S$. 
In general,  if $N(u)\cap S\not= N(v)\cap S$, we say that $S$  \emph{distinguishes} the pair $u$ and $v$.
A \emph{locating-dominating set}, \emph{LD-set} for short, is a distinguishing set that is also dominating. 
Observe that there is at most one vertex not dominated by a distinguishing set.
The \emph{location-domination number} of $G$,   denoted by $\lambda (G)$, is the  minimum cardinality of a locating-dominating set.
A locating-dominating set of cardinality $\lambda(G)$ is called an \emph{LD-code}~\cite{rasl84,slater88}.
%We say that wo vertices $u$ and $v$ are \emph{twins} if  $N(u)=N(v)$ or $N(u)\cup \{ u \}=N(v)\cup \{v \}$.
%The \emph{location-domination number} of $G$,   denoted by $\lambda (G)$, is the  minimum cardinality of a locating-dominating set.
Certainly,  every LD-set of a non-connected graph $G$ is  the  union of LD-sets of its connected components and the location-domination number is the sum of the location-domination number of its connected components.
Both, LD-codes and the location-domination parameter have been intensively studied during the last decade;
see \cite{bchl,bcmms07,cahemopepu12,clm11,fhls16,ours3,oursglobal,hola06}.
A complete and regularly updated list of papers on locating-dominating codes is to be found in \cite{lobstein}.

The \emph{complement} of $G$,  denoted by $\overline{G}$, has the same set of vertices of $G$ and two vertices are adjacent in $\overline{G}$ if and only if they are not adjacent in $G$. 
This work is devoted to approach the relationship between $\lambda (G)$ and $\lambda (\overline{G})$ for connected bipartite graphs.

It follows immediately from the definitions that a set $S\subseteq V$ is distinguishing in $G$ 
if and only if it is distinguishing in $\overline{G}$.  
A straightforward consequence of this fact are the following results.

\vspace{.3cm}
%%%%%%%%%%%%%%%%%%%%%%%%%%%%%%%%%%%%%%%%%%%%%%%%%%%%%
\begin{proposition} [\cite{oursglobal}]\label{pro.domi}
	Let $S\subseteq V$ be an LD-set of a graph $G=(V,E)$.
	Then,  $S$ is an LD-set of $\overline{G}$ if and only if $S$ is a dominating set of $\overline{G}$;
\end{proposition}

\vspace{.3cm}
%%%%%%%%%%%%%%%%%%%%%%%%%%%%%%%%%%%%%%%%%%%%%%%%%%%%%
\begin{proposition} [\cite{ours3}]\label{pro.vertexdom}
	Let  $S\subseteq V$ be an LD-set of a graph $G=(V,E)$. Then, the following properties hold.
	\begin{itemize}
		\item[(a)]  There is at most one vertex $u\in V\setminus S$ such that $N(u)\cap S=S$, and in the case it exists, $S\cup \{ u \}$ is an LD-set of $\overline{G}$.
		\item[(b)] $S$ is an LD-set of $\overline{G}$ if and only if there is no vertex in $V\setminus S$ such that $N(u)\cap S=S$.
	\end{itemize} 
\end{proposition}

\vspace{.3cm}
%%%%%%%%%%%%%%%%%%%%%%%%%%%%%%%%%%%%%%%%%%%%%%%%%%%%%
\begin{theorem} [\cite{ours3}]\label{cor.difuno}
	For every graph $G$, $|\lambda (G) -\lambda (\overline{G})|\le 1$.
\end{theorem}
%%%%%%%%%%%%%%%%%%%%%%%%%%%%%%%%%%%%%%%%%%%%%%%%%%%%%

According to the preceding inequality, $\lambda (\overline{G})\in\{\lambda (G)-1,\lambda (G),\lambda (G)+1\}$  for every graph $G$, all cases being  feasible for some connected graph $G$.
We intend to determine graphs such that $\lambda (\overline{G})>\lambda (G)$, that is, we want to solve the equation $\lambda (\overline{G})= \lambda (G) +1$.
This problem was completely solved in \cite{oursglobal} for the family of block-cactus.

In this work, we carry out a similar study for bipartite graphs. 
For this purpose, we first introduce in  Section~\ref{sec:associatedGraph} the graph associated with a distinguishing set. 
This graph turns out to be very helpful to derive  some properties related to LD-sets and the location-domination number of $G$, and will be used to get the main results in Section~\ref{sec:bipartite}.

In Table \ref{tab.valors},  the location-domination number of some families of bipartite graphs are displayed, along with the location-domination number of its complement graphs.
Concretely, we consider the path $P_n$ of order  $n\ge4$; the cycle $C_n$ of (even) order  $n\ge4$;  
%the wheel  $W_n$  of order  $n\ge8$, obtained by joining a new vertex to all vertices of a cycle of order $n-1$; the complete graph  $K_n$ of order  $n\ge2$; 
the star $K_{1,n-1}$ of order  $n\ge4$, obtained by joining a new vertex to $n-1$ isolated vertices; 
the complete bipartite graph $K_{r,n-r}$ of order  $n\ge4$, with $2\le r\le n-r$ and stable sets of order $r$ and $n-r$, respectively; 
and finally, the bi-star $K_2(r,s)$ of order $n\ge6$ with  $3\le r\le s=n-r$,  obtained by joining the central vertices of two stars $K_{1,r-1}$ and $K_{1,s-1}$ respectively.

\vspace{.2cm}
%%%%%%%%%%%%%%%%%%%%%%%
%%%%%%%%%%%%%%%%%%%%%%%
\begin{proposition} [\cite{oursglobal}]\label{donosti}
	Let   $G$ be a graph of order $n\ge 4$.
	If $G$ is a graph belonging to one of the following classes: $P_n,C_n,K_{1,n-1},K_{r,n-r}, K_2(r,s)$,
	then the values of $\lambda (G)$ and $\lambda (\overline{G})$ are known and they are displayed in Table \ref{tab.valors}.
\end{proposition}
%%%%%%%%%%%%%%%%%%%%%%%
%%%%%%%%%%%%%%%%%%%%%%%

\vspace{.7cm}{\large
%%%%%%%%%%%%%%%%%%%%%%%%%%%%%%
%%%%%%%%%%%%%%%%%%%%%%%%%%%%%%
\begin{table}[hbt]
	\begin{center}
		\begin{tabular}{c||cc|cc}
			{ $G$} & { $P_n$} & { $P_n$} &  { $C_n$} &   { $C_n$}  \\
			\hline
			{ $n$} &  $4\le n\le 6$ & { $n\ge 7$} & $4\le n\le 6$  & { $n\ge 7$} \\
			\hline			
			{ $\lambda (G)$} & { $\lceil \frac {2n}5 \rceil$} & { $\lceil \frac {2n}5 \rceil$}   &{ $\lceil \frac {2n}5 \rceil$}    & { $\lceil \frac {2n}5 \rceil$}  \\
			%&&&&&&&\\
			%
			{ $\lambda (\overline{G})$} & { $\lceil \frac {2n}5 \rceil$} & { $\lceil \frac {2n-2}5 \rceil$} & { $\lceil \frac {2n}5 \rceil$} & { $\lceil \frac {2n-2}5 \rceil$}  \\
			
			%&&&&&&&\\
		\end{tabular}
	\end{center}
	
	\begin{center}
		\begin{tabular}{c||ccc}
			{ $G$} & { $K_{1,n-1}$} &   { $K_{r,n-r}$} &   {$K_2${$(r,s)$}} \\
			\hline
			{ $n$} &   { $n\ge 4$} &  { $2\le r\le n-r$}&  { $3\le r\le s$} \\
			\hline			
			{ $\lambda (G)$} &   { $n-1$}  &  { $n-2$}  &  { $n-2$} \\
			%&&&&&&&\\
			%
			{ $\lambda (\overline{G})$} &   { $n-1$}  &  { $n-2$} &  { $n-3$} \\
			
			%&&&&&&&\\
		\end{tabular}
	\end{center}
	\caption{The values of  $\lambda (G)$ and $\lambda (\overline{G})$ for some families of bipartite graphs.}
	\label{tab.valors}
\end{table}
}
%%%%%%%%%%%%%%%%%%%%%%%%%%%%%%
%%%%%%%%%%%%%%%%%%%%%%%%%%%%%%

Notice that in all cases considered in Proposition~\ref{donosti}, we have $\lambda (\overline{G})\le \lambda (G)$.
Moreover, observe also that, for every pair of integers $(r,s)$ with  $3\le r\le s$, we have examples of bipartite graphs with stable sets of order $r$ and $s$ respectively, such that $\lambda (\overline{G})= \lambda (G)$ and such that $\lambda (\overline{G})= \lambda (G)-1$.
%\merce{Comprovar si es pot posar ordre m\'es petit en alguns casos de la taula.}

\newpage
%%%%%%%%%%%%%%%%%%%%%%%%%%%%%%%%%%%%%%%%%%%%%%%%%%%%%
%%%%%%%%%%%%%%%%%%%%%%%%%%%%%%%%%%%%%%%%%%%%%%%%%%%%%
%%%%%%%%%%%%%%%%%%%%%%%%%%%%%%%%%%%%%%%%%%%%%%%%%%%%%
%%%%%%%%%%%%%%%%%%%%%%%%%%%%%%%%%%%%%%%%%%%%%%%%%%%%%
%%%%%%%%%%%%%%%%%%%%%%%%%%%%%%%%%%%%%%%%%%%%%%%%%%%%%
%%%%%%%%%%%%%%%%%%%%%%%%%%%%%%%%%%%%%%%%%%%%%%%%%%%%%

\section{The graph  associated with a distinguishing set}\label{sec:associatedGraph}

Let $S $ be a distinguishing set of a graph $G$.
We introduce in this section a labeled graph associated with $S$ and study some general properties.
Since LD-sets are distinguishing sets that are also dominating, this graph allows us to derive  some properties related to LD-sets and the location-domination number of $G$.

\begin{defi}\label{def.Gomega} \rm
	Let $S$ be  a distinguishing set of cardinality $k$ of a graph $G=(V,E)$ of order $n$.
	The so-called \emph{S-associated graph}, denoted by $G^{S}$, is the edge-labeled graph defined as follows.
	\begin{itemize}
		\item[i)] $V(G^{S})=V\setminus S$;
		\item[ii)]  If $x,y\in V(G^{S})$, then $xy\in E(G^{ S})$ if and only if the sets of neighbors of $x$ and $y$ in $S$ differ in exactly one vertex $u(x,y)\in S$;
		\item[iii)]  The label $\ell (xy)$ of edge $xy\in E(G^{ S})$ is the only vertex $u(x,y)\in S$ described in the preceding item. 
	\end{itemize}
\end{defi}

%%%%%%%%%%%%%%%%%%%%%%%%
%%%%%%%%%%%%%%%%%%%%%%%%
\begin{figure}[!hbt]
	\begin{center}
		\includegraphics[width=0.35\textwidth]{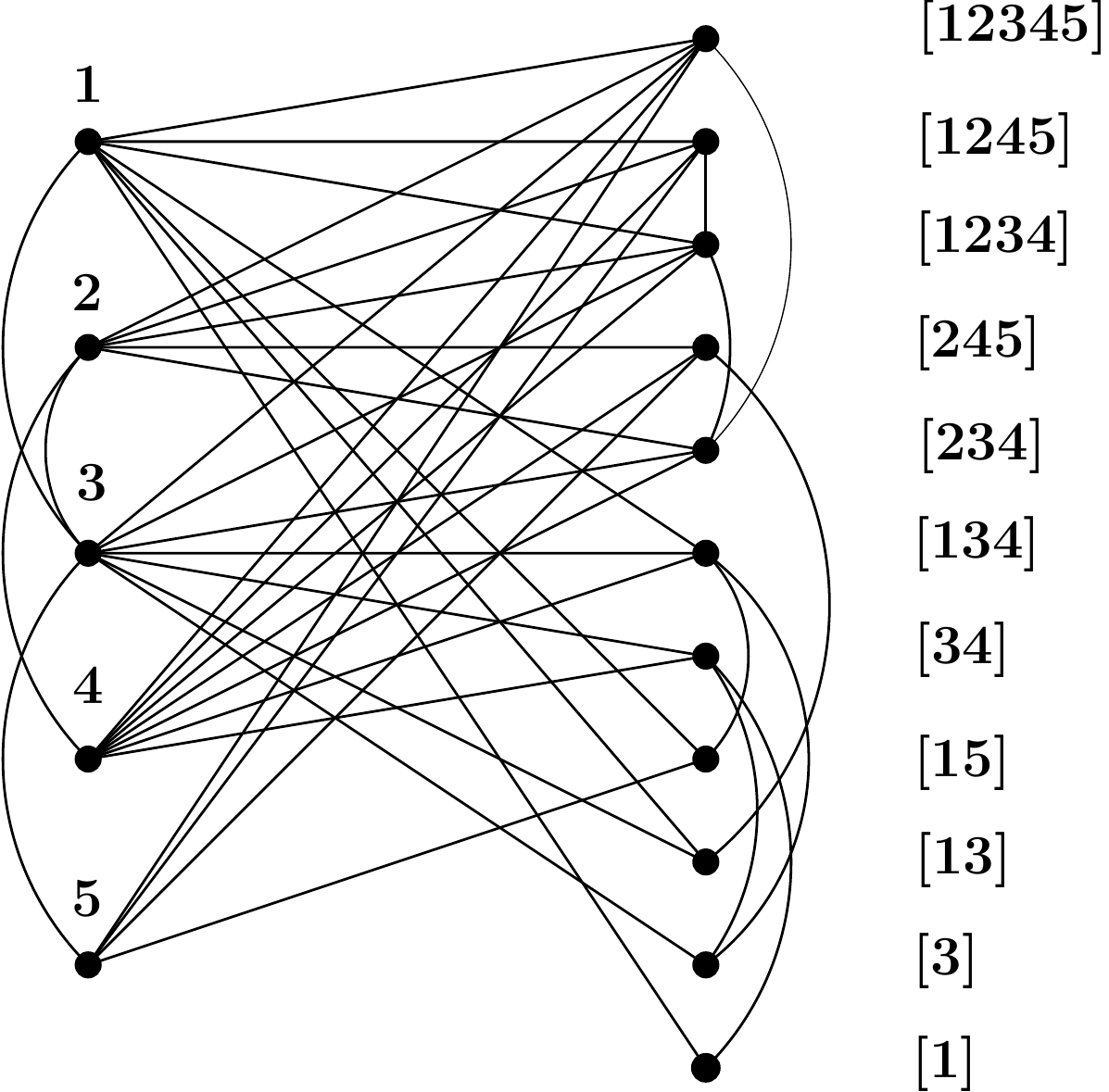}\hspace{2cm} \includegraphics[width=0.3\textwidth]{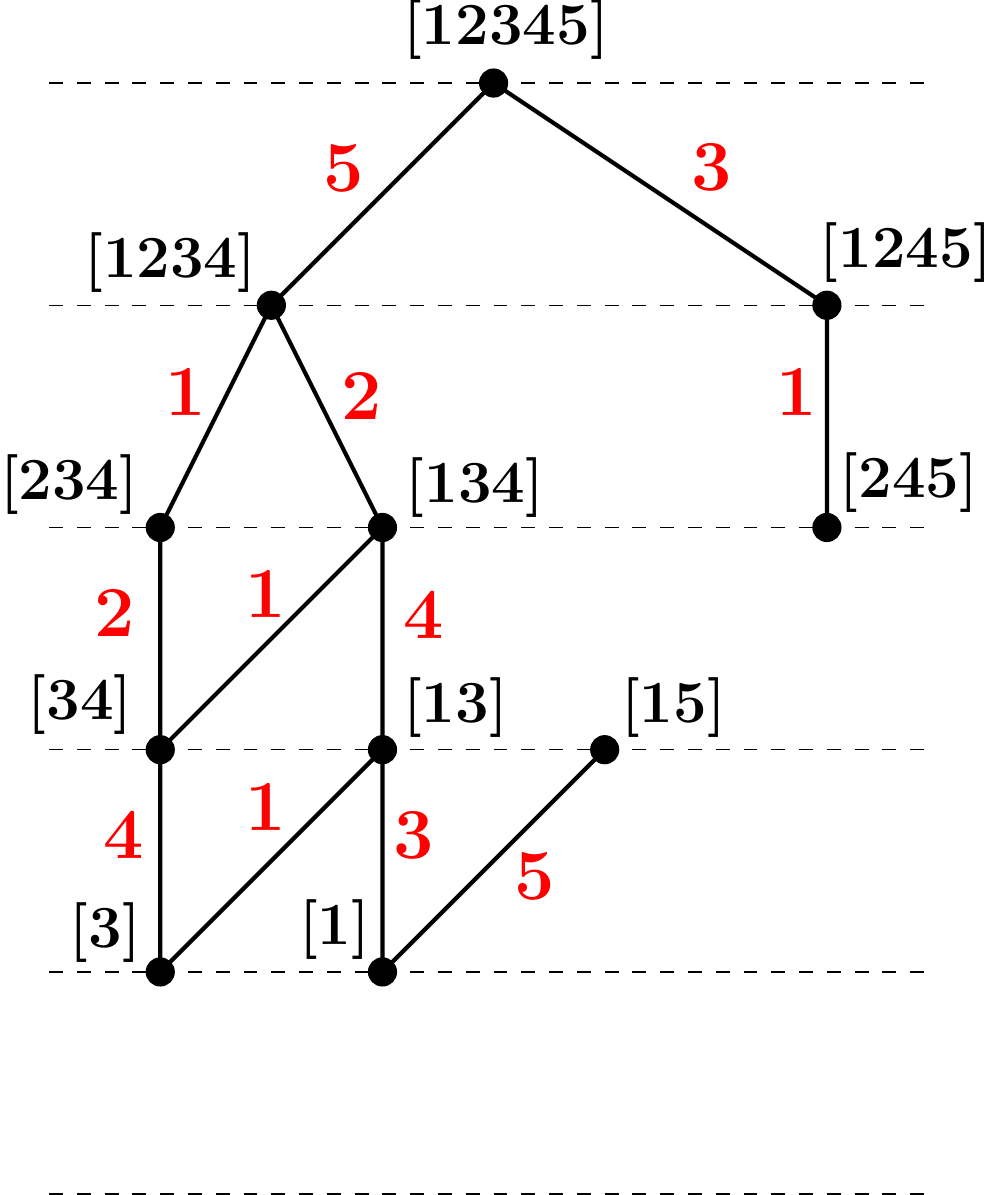}
	\end{center}
	\caption{A graph $G$ (left) and the graph $G^{S}$ associated with the distinguishing set $S=\{ 1,2,3,4,5 \}$ (right). The neighbors in $S$ of each vertex are those enclosed in brackets.}
	\label{fig:exempleGestrella}
\end{figure}
%%%%%%%%%%%%%%%%%%%%%%%%
%%%%%%%%%%%%%%%%%%%%%%%%

Notice that if $xy \in E(G^{S})$, $\ell (xy)=u\in S$ and $|N(x)\cap S|>|N(y)\cap S|$, then $N(x)\cap S=(N(y)\cap S) \cup \{u\}$.
Therefore, we can represent the graph $G^{S}$ with the vertices lying on $|S |+1=k+1$ levels, from bottom (level $0$) to top (level $k$), in such a way that vertices with exactly $j$ neighbors in $S$ are at level $j$. 
For any $j\in \{ 0,1,\dots , k \}$ there are at most $\binom {k} {j}$ vertices at level $j$.
So, there is at most one vertex at level $k$ and, if it is so, this vertex is adjacent to all vertices of $S$. 
There is at most one vertex at level $0$ and, if it is so, this vertex has no neighbors in $S$.
Notice that $S$ is an LD-set if and only if there is no vertex at level $0$.
The vertices at level 1 are those with exactly one neighbor in $S$. 
See Figure \ref{fig:exempleGestrella} for an example of an LD-set-associated graph.

Next, we state some basic  properties of the graph associated with a distinguishing set that will be used later.

%\newpage
%%%%%%%%%%%%%%%%%%%%%%%%
%%%%%%%%%%%%%%%%%%%%%%%%
\begin{proposition}\label{pro:basic}
	Let $S$ be a distinguishing set of $G=(V,E)$, $x,y\in V\setminus S$ and $u\in S$. Then,
	
	\begin{enumerate}[{\rm (1)}]
		
		\item $S$ is a distinguishing set of $\overline{G}$.
		
		\item The associated graphs $G^S$ and $\overline{G}^S$ are equal.
		
		\item  The representation by levels of  $\overline{G}^S$ is obtained by reversing bottom-top the representation of  $G^S$.
		
		\item $xy \in E(G^{S})$ and  $\ell (xy)=u$ if and only if   $x$ and $y$ have the same neighborhood in $S \setminus \{ u \}$ and (thus) they are not distinguished by $S\setminus \{ u \}$.
		
		\item If $xy \in E(G^{S})$ and  $\ell (xy)=u$, then $S\setminus \{ u \}$ is not a distinguishing set.

	\end{enumerate}
	
\end{proposition}
%%%%%%%%%%%%%%%%%%%%%%%%
%%%%%%%%%%%%%%%%%%%%%%%%

%%%%%%%%%%%%%%%%%%%%%%%%
%%%%%%%%%%%%%%%%%%%%%%%%
\begin{figure}[!hbt]
	\begin{center}
		\includegraphics[height=3.8cm]{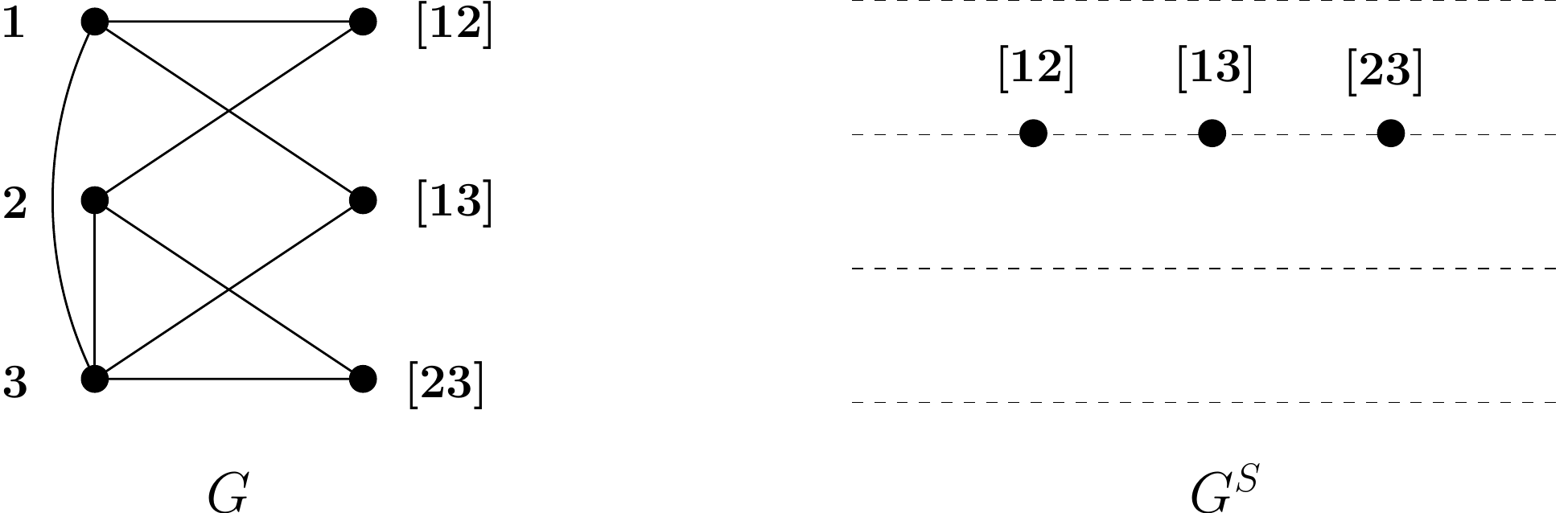}
		\caption{$S=\{1,2,3\}$ is distinguishing, $S'=\{1,2\}$ is not distinguishing  and $G^S$ has no edges.}
		\label{fig:contra}
	\end{center}
\end{figure}
%%%%%%%%%%%%%%%%%%%%%%%%
%%%%%%%%%%%%%%%%%%%%%%%%

The converse of Proposition~\ref{pro:basic} (5) is not necessarily true. 
For example, consider the graph $G$ of order 6 displayed in Figure \ref{fig:contra}. 
By construction, $S=\{1,2,3\}$ is a distinguishing set. 
However, $S'=S\setminus \{ 3\}=\{1,2\}$ is not a distinguishing set, because $N(3)\cap S'=N([12])\cap S'=\{ 1,2\}$, and the $S$-associated graph $G^{S}$ has no edge with label $3$ (in fact,  $G^{S}$ has no edges since the neighborhoods in $S$ of all vertices not in $S$ have the same size).

As a straight consequence of Proposition~\ref{pro:basic} (5), the following result is derived.

%%%%%%%%%%%%%%%%%%%%%%%%%%%
%%%%%%%%%%%%%%%%%%%%%%%%%%%
\begin{corollary}\label{pro:suprimirVarios}
	Let $S$ be a distinguishing set of $G$ and let $S'\subseteq S$.  Consider the subgraph $H_{S'}$ of $G^S$ induced by the edges with a label from $S'$. 
	Then, all the vertices belonging to the same connected component in $H_{S'}$ 
	have the same neighborhood  in $S\setminus S'$, concretely, it is the neighborhood in $S$ of a vertex lying on the lowest level.
\end{corollary}
%%%%%%%%%%%%%%%%%%%%%%%%%%
%%%%%%%%%%%%%%%%%%%%%%%%%%

For example, consider the graph shown in Figure~\ref{fig:exempleGestrella}.
If $S'=\{1,2\}$, then vertices of the same connected component in $H_{S'}$ have the same neighborhood in $S\setminus S'$. Concretely, the neighborhood of vertices [1234], [234], [134] and [34] in $S\setminus \{1,2\}$ is $\{3,4\}$;
the neighborhood of vertices [13] and [3] in $S\setminus \{1,2\}$ is $\{3\}$; and 
the neighborhood of vertices [1245] and [245] in $S\setminus \{1,2\}$ is $\{4,5\}$
(see Figure~\ref{fig:components}).

%%%%%%%%%%%%%%%%%%%%%%%%
%%%%%%%%%%%%%%%%%%%%%%%%
\begin{figure}[!hbt]
	\begin{center}
		\includegraphics[height=4.5cm]{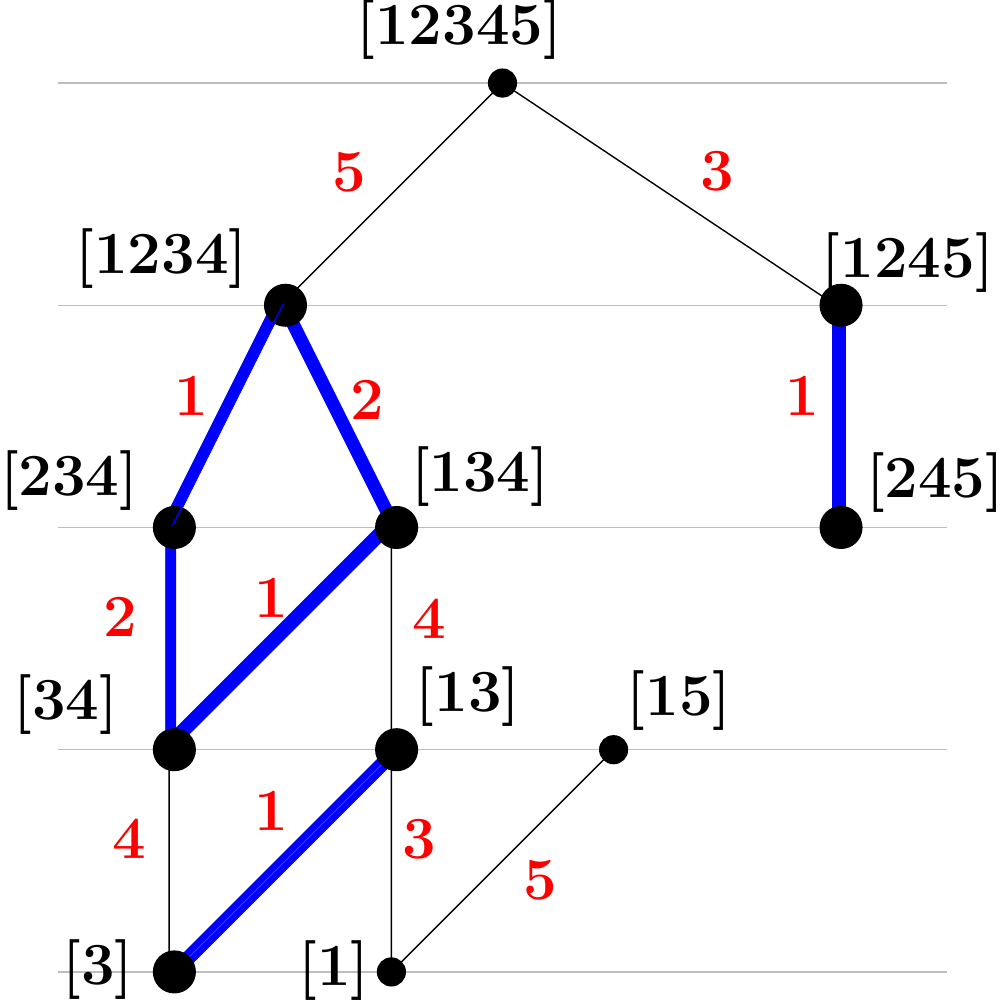}
		\caption{If $S'=\{1,2\}$, then $H_{S'}\cong C_4+2K_2$ has three components. Vertices of the same  component in $H_{S'}$ have the same neighborhood in $S\setminus S'$.}
		\label{fig:components}
	\end{center}
\end{figure}
%%%%%%%%%%%%%%%%%%%%%%%%
%%%%%%%%%%%%%%%%%%%%%%%%

\begin{proposition}\label{claims123}
	Let $S$ be  a distinguishing set of cardinality $k$ of a connected graph $G$ of order $n$.
	Let $G^{S}$ be its  associated graph. Then, the following conditions hold.
	\begin{enumerate}[{\rm (1)}]
		\item {$|V(G^{S})|=n-k$.}
		
		\item $G^{S}$ is bipartite.
		
		\item Incident edges of $G^S$ have different labels.
		
		\item Every cycle of $G^{S}$ contains an even number of edges labeled $v$, for all $v\in S$.
		
		\item Let $\rho$  be a walk with no repeated edges in $G^{S}$.
		If  $\rho$ contains an even number of edges labeled $v$ for every $v\in S$, then $\rho$  is a closed walk.
		
		\item If $\rho =x_ix_{i+1}\dots x_{i+h}$ is a path satisfying that vertex $x_{i+h}$ lies at level $i+h$, for any $h\in \{ 0,1,\dots ,h \}$, then
		\begin{enumerate}
			\item
			the edges of $\rho$ have different labels;
			%all the labels of the edges of $\rho$  are different;
			%$\ell (x_ix_{i+1})$,   $\ell (x_{i+1}x_{i+2})$, $\dots $, $\ell (x_{i+h-1}x_{i+2h})$
			\item for all  $j\in \{ i+1,i+2,\dots ,i+h \}$, $N(x_j)\cap S$ contains the vertex $\ell (x_kx_{k+1})$, for any $k\in \{ i,i+1,\dots , j-1\}$.
		\end{enumerate}

	\end{enumerate}
\end{proposition}
\begin{proof}
	\begin{enumerate}[{\rm (1)}]
		\item It is a direct consequence  from the definition of $G^{S}$.
		
		\item Take  $V_1=\{ x\in V(G^{S}) : |N(x)\cap S |\textrm{ is odd}\}$ and  $V_2=\{ x\in V(G^{S}) : |N(x)\cap S|\textrm{ is even}\}$.
		Then, $V(G^{S})=V_1\cup V_2$ and $V_1\cap V_2 =\emptyset$.
		Since $||N(x)\cap S|-|N(y)\cap S||=1$ for any $xy\in E(G^{S})$, it is clear that the vertices $x,y$ are not in the same subset $V_i$, $i=1,2$.
		
		\item Suppose that edges $e_1=xy$  and $e_2=yz$ have the same label $l(e_1)=l(e_2)=v$.
		This means that  $N(x)\cap S$ and $N(y)\cap S$ differ only in vertex $v$, and  $N(y)\cap S$ and $N(z)\cap S$ differ only in vertex $v$.
		It is only possible if $N(x)\cap S=N(z)\cap S$, implying that $x=z$.
		
		\item Let $\rho$ be a cycle such that $E(\rho)=\{x_0x_1,x_1x_2,\ldots x_hx_0\}$.
		The set of neighbors in $S$ of two consecutive vertices differ exactly in one vertex.
		If we begin with $N(x_0)\cap S$, then each time we add (remove) the vertex of the label of the corresponding edge, we have to remove (add) it later in order to obtain finally the same neighborhood,  $N(x_0)\cap S$.
		Therefore, $\rho$ contains an even number of edges with label $v$.
		
		\item
		Consider the vertices $x_0,x_1,x_2,x_3,...,x_{2k}$ of  $\rho$.
		In this case, $N(x_{2k})\cap S$ is obtained from $N(x_0)\cap S$ by either adding or removing  the labels of all the edges of the walk.
		As every label appears an even number of times, for each element $v\in S$ we can match its appearances in pairs, and each pair means that we  add and remove (or remove and add) it from the neighborhood in $S$.
		Therefore, $N(x_{2k})\cap S=N(x_0)\cap S$, and hence $x_0=x_{2k}$.
		
		\item It straightly follows  from the fact that $N(x_{j})\cap S = (N(x_{j-1})\cap S)\cup \{ \ell (x_{j-1}x_{j})\}$, for any $j\in \{ i+1,\dots ,i+h\}$.
	\end{enumerate}
	\vspace{-.9cm}\end{proof}

%%%  def de cactus aquí
In the study of distinguishing sets and LD-sets using its associated graph, a family of graphs is particularly useful, the \emph{cactus graph} family.
A \emph{block} of a graph is a maximal connected subgraph with no cut vertices.
A connected graph $G$ is a \emph{cactus} if all its blocks are either cycles or edges.
Cactus are characterized as those connected graphs with no edge shared by two cycles.

\begin{lemma}\label{lem:claim.subgraphcactus}
	%Let $G=(S\cup W,E)$ be a bipartite graph such that $r\le|S|\le|W|$ and $\lambda (\overline{G})= \lambda (G)+1$.
	%Let $\lambda (\overline{G})= \lambda (G)+1$
	%and assume that $S$ is an LD-code of $G$. 
	Let $S$ be a distinguishing set of a graph $G$ and $\emptyset\not= S'\subseteq S$.
	Consider a subgraph $H$ of $G^S$ induced by a set of edges containing exactly two edges with label $u$, for each $u\in S'\subseteq S$.
	Then, all the connected components of $H$ are cactus.
\end{lemma}

\begin{proof}
	We  prove that there is no edge lying on two different cycles of $H$.
	Suppose, on the contrary, that there is an edge $e_1$ contained in two different cycles $C_1$ and $C_2$ of $H$.
	Note that $C_1$ and $C_2$ are cycles of $G^S$, since $S'\subseteq S$.
	Hence, if the label of $e_1$ is $u\in S'\subseteq S$, then by Proposition \ref{claims123} both cycles $C_1$ and $C_2$ contain the other edge $e_2$ of $H$ with label $u$.
	Suppose that $e_1=x_1y_1$ and $e_2=x_2y_2$ and assume without loss of generality that there exist $x_1-x_2$ and $y_1-y_2$ paths in $C_1$ not containing edges $e_1,e_2$. Let $P_1$ and $P_1'$ denote respectively those paths (see Figure \ref{fig.ac} a).
	
	We have two possibilities for $C_2$:
	(i) there are  $x_1-x_2$ and $y_1-y_2$ paths in $C_2$ not containing neither $e_1$ nor $e_2$. Let $P_2$ denote the $x_1-x_2$ path in $C_2$ in that case (see Figure \ref{fig.ac} b);
	(ii) there are  $x_1-y_2$ and $y_1-x_2$ paths in $C_2$ not containing neither $e_1$ nor $e_2$ (see Figure \ref{fig.ac} c).
	
	In  case (ii), the closed walk formed with the  path $P_1$, $e_1$ and the $y_1-x_2$ path in $C_2$ would contain a cycle with exactly one edge labeled with $u$, a contradiction (see Figure \ref{fig.ac} d).
	
	In case (i), at least one the following cases hold: either the $x_1-x_2$ paths in $C_1$ and in $C_2$, $P_1$ and $P_2$, are different, or the $y_1-y_2$ paths in $C_1$ and in $C_2$ are different (otherwise, $C_1=C_2$).

	%%%%%%%%%%%%%%%%%%%%%
	%%%%%%%%%%%%%%%%%%%%%
	\begin{figure}[!hbt]
		\begin{center}
			\includegraphics[height=6cm]{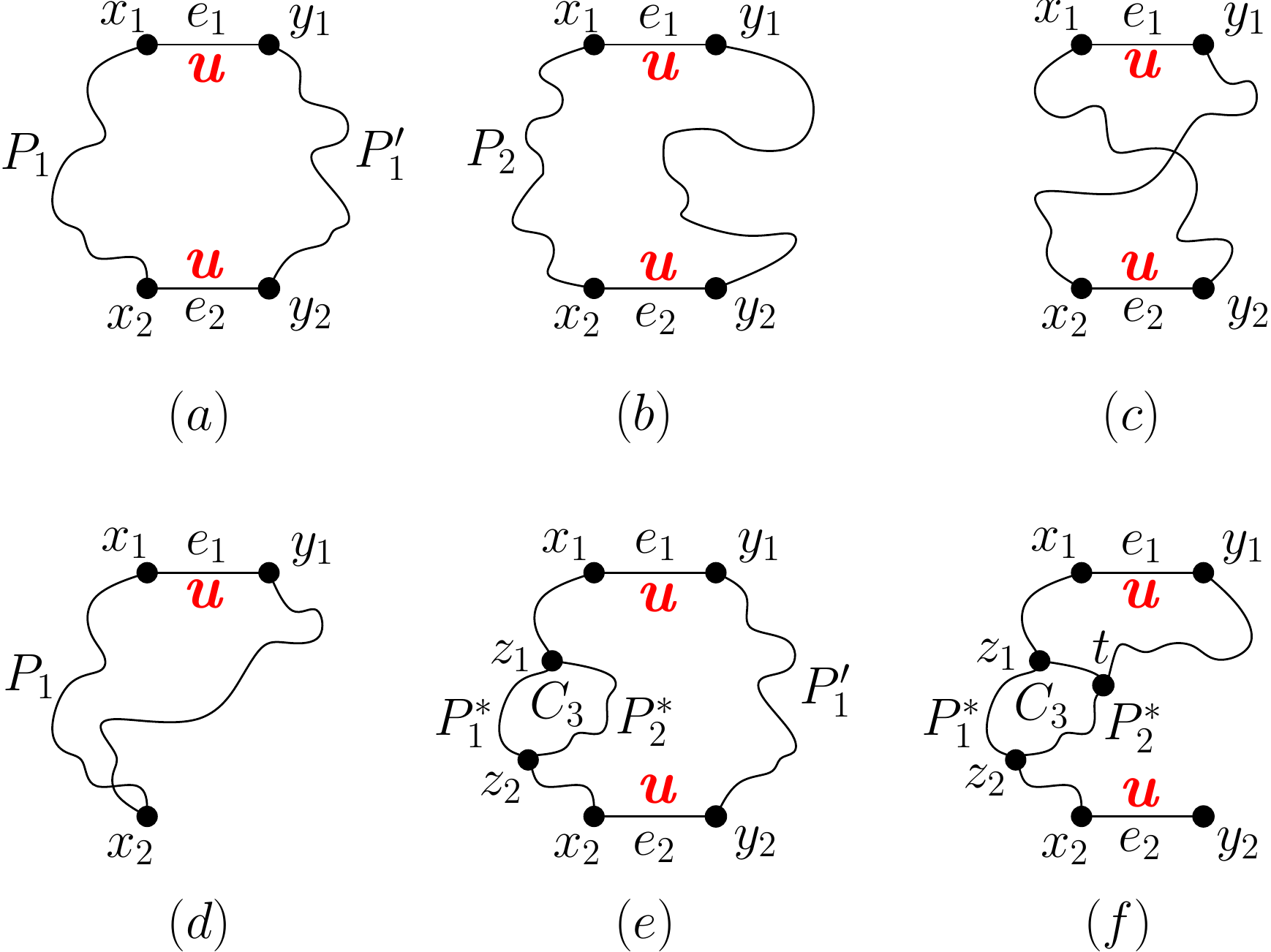}
			\caption{All connected components of the subgraph $H$ are cactus.}\label{fig.ac}
		\end{center}
	\end{figure}
	%%%%%%%%%%%%%%%%%%%%%
	%%%%%%%%%%%%%%%%%%%%%

	%\merce{Nova demostració}
	Assume that $P_1$ and  $P_2$ are different.
	Let $z_1$ be the last vertex shared by $P_1$ and $P_2$ advancing from $x_1$ and let $z_2$ be the first vertex shared by $P_1$ and $P_2$ advancing from $z_1$ in $P_2$.
	Notice that $z_1\not= z_2$.
	Take the cycle $C_3$ formed with the $z_1-z_2$ paths in $P_1$ and $P_2$.
	Let $P_1^*$ and $P_2^*$ be respectively the $z_1-z_2$  subpaths of $P_1$ and $P_2$ (see Figure \ref{fig.ac} e).
	We claim that the internal vertices of $P_2^*$ do not lie in $P_1'$.
	Otherwise, consider the first vertex $t$ of $P_1'$ lying also in $P_2^*$.
	The cycle beginning in $x_1$, formed by the edge $e_1$, the $y_1-t$  path contained in $P_1'$, the $t-z_1$ path contained in $P_2^*$ and the $z_1-x_1$ path contained in $P_1$ has exactly one appearance of an edge with label $u$, which is a contradiction
	(see Figure \ref{fig.ac} f).
	By Proposition \ref{claims123}, the labels of edges belonging to  $P_1^*$  appear an even number of times in cycle $C_3$, but they also appear an even number of times in cycle $C_1$.
	But this is only possible if they appear exactly two times in $P_1^*$, since $H$ contains exactly two edges with the same label.
	By Proposition \ref{claims123}, $P_1^*$ must be a  closed path, which is a contradiction.
\end{proof}

Next,  we establish a relation between some parameters of bipartite graphs having cactus as connected components.
We denote by $\cc (G)$ the number of connected components of a graph $G$.

\begin{lemma}\label{lemacactus}
	Let $H$ be a bipartite graph of order at least 4 such that all its connected components are cactus. 
	Then, $|V(H)|\ge \frac 34 |E(H)|+\cc(H)\ge \frac 34 |E(H)|+1$.
\end{lemma}

\begin{proof}
	Let $\cy(H)$ denote the  number of cycles of $H$. 
	Since $H$ is a planar graph with $\cy(H)+1$ faces and $\cc(H)$  components, the equality follows from the generalization of Euler's Formula:
	$$(\cy(H)+1)+|V(H)|=|E(H)|+(\cc(H)+1).$$
	Let $\ex (H)=|E(H)|-4\, \cy(H)$.  Then,
	\begin{align*}
	|V(H)|&=|E(H)|-\cy(H)+\cc (H)%\\ &
	=|E(H)|-\frac 14 (|E(H)-\ex (H))+\cc (H)\\
	&=\frac 34 |E(H)|+\frac 14 \ex (H) +\cc (H).
	\end{align*} 
	But $\cc (H)\ge 1$,  and  $\ex (H)\ge 0$ as all cycles of a bipartite graph have at least 4 edges. Thus,
	$$|V(H)|=
	\frac 34 |E(H)|+\frac 14 \ex (H) +\cc (H) \ge
	\frac 34 |E(H)|+\cc (H) \ge  \frac 34 |E(H)| +1.$$
	
\vspace{-1.3cm}\end{proof}

\vspace{.4cm}
%%%%%%%%%%%%%%%%%%%%%%%%%%%%%%%%%
%%%\begin{corollary}\label{cor:claim.subgraphcactus}
%%%	Let $S$ be a distinguishing set of a graph $G$ and $\emptyset\not= S'\subseteq S$.
%%%	Consider a subgraph $H_{S'}$ of $G^S$ induced by a set of edges containing exactly two edges with label $u$ for each $u\in S'\subseteq S$.
%%%	Let $r'=|S'|$. Then, $|V(H_{S'})|\ge \frac 3 2 r'+1$.
%%%\end{corollary}
%%%%%%%%%%%%%%%%%%%%%%%%%%%%%%%%%
%%%%\begin{proof} Aplicar lema anterior.
%%%%\end{proof}

The preceding result allows us to give a lower bound of the order of some subgraphs of the graph associated with a distinguishing set.

\begin{corollary}\label{cor:claim.subgraphcactus}
Let $S$ be a distinguishing set of a graph $G$ and $\emptyset \not= S'\subseteq S$.
Consider a subgraph $H$ of $G^S$ induced by a set of edges containing exactly two edges with label $u$ for each $u\in S'\subseteq S$.
Let $r'=|S'|$. Then, $|V(H)|\ge \frac 3 2 r'+1$.
\end{corollary}
%%%%%%%%%%%%%%%%%%%%%%%%%%%%%%
\begin{proof} Since two edges of $G^S$ with the same label have no common endpoints, we have $|V(H)|\ge 4$, and the result follows by applying Lemmas~\ref{lem:claim.subgraphcactus} and \ref{lemacactus} to $H$.
\end{proof}

Next lemma states a property about the difference of the order and the number of connected components of a subgraph and will be used to prove the main result of this work. 

\begin{lemma}\label{componentsSubgraf} 
If $H$ is a subgraph of $G$, then $|V(G)|-\cc (G)\ge |V(H)|-\cc (H)$.
\end{lemma}
\begin{proof} Since every subgraph of $G$ can be obtained by successively removing vertices and edges from $G$, it is enough to prove that the inequality holds whenever a vertex or an edge is removed from $G$.

Let $u\in V(G)$. If $u$ is an isolated vertex in $G$, then 
the order and the number of components decrease in exactly one unit when removing $u$ from $G$, so that the given inequality holds.
If $u$ is a non-isolated vertex, then the order decreases in one unit while the number of components does not decrease when removing $u$ from $G$. Thus, the given inequality holds.

Now let $e\in E(G)$. Notice that the order does not change when removing an edge from $G$.  If $e$ belongs to a cycle, then the number of components does not change when removing $e$ from $G$, and the given inequality holds. If $e$ does not belong to a cycle, then the  the number of components increases in exactly one unit when removing $e$ from $G$, and the given inequality holds.
\end{proof}

\newpage
%%%%%%%%%%%%%%%%%%%%%%%%%%%%%%%%%%%%%%%%%%%%%%%%%%%%%
%%%%%%%%%%%%%%%%%%%%%%%%%%%%%%%%%%%%%%%%%%%%%%%%%%%%%%%
%%%%%%%%%%%%%%%%%%%%%%%%%%%%%%%%%%%%%%%%%%%%%%%%%%%%%%%
\section{The bipartite case}\label{sec:bipartite}

This section is devoted to solve the equation $\lambda (\overline{G}) =\lambda (G)  + 1$  when we restrict ourselves to bipartite graphs.
In the sequel,  $G=(V,E)$ stands for  a bipartite connected graph of order $n=r+s\ge4$, such that  $V=U\cup W$, where $U$ and $W$ are the stable sets and  $$1\le |U|=r\le s=|W|.$$

In the study of LD-sets, vertices with the same neighborhood play an important role, since at least one of them must be in an LD-set.
We say that two vertices $u$ and $v$ are \emph{twins} if  either $N(u)=N(v)$ or $N(u)\cup \{ u \}=N(v)\cup \{v \}$.

\begin{lemma}\label{prop.SW} Let $S$ be an LD-code of
	$G$. Then, $\lambda (\overline{G})\le \lambda (G)$ if any of the following conditions hold.
	\begin{enumerate}[i)]
		\item $S\cap U\not= \emptyset$ and $S\cap W\not= \emptyset$.
		\item $r<s$ and  $S=W$.
		\item $2^r\le s$.
	\end{enumerate}
\end{lemma}
\begin{proof}
	If $S$ satisfies item i), then
	the LD-code of $G$ is a distinguishing set of $\overline{G}$ and it is dominating in $\overline{G}$ because there is no vertex in $G$
	with neighbors in both stable sets. 
	Thus, $\lambda (\overline{G})\le \lambda (G)$.
	
	Next, assume that  $r<s$ and  $S=W$. 	
	In this case, $\lambda(G)=|W|>|U|$ and thus $U$ is not an LD-set, but it is a dominating set since $G$ is connected. Therefore, there exists a pair of  vertices $w_1,w_2\in W$ such that $N(w_1)=N(w_2)$. Hence,  $W-\{ w_1 \}$ is an LD-set of $G-w_1$. Let $u\in U$ be a vertex adjacent to $w_1$ (it exists since $G$ is connected), and notice that
	$(W\setminus \{ w_1 \})\cup \{ u \}$ is an LD-code of $G$ with vertices in both stable sets, which,  by the preceding item, means that $\lambda (\overline{G})\le \lambda (G)$.
	
	Finally, if  $2^r\le s$ then  $S\neq U$, which means that $S$ satisfies either item i) or item ii).
\end{proof}

%%%%%%%%%%%%%%%%%%%%%%%%%%%%%%%%%%%%%%%%%%%%%%%%%%%%%
\begin{proposition}\label{pro.r1o2}
	If $G$ has order at least 4 and $1\le r \le 2$, then $\lambda (\overline{G}) \le \lambda (G)$.
\end{proposition}
\begin{proof}
	If $r=1$, then $G$ is the star $K_{1,n-1}$ and $\lambda (\overline{G})=\lambda (G)=n-1$.

	Suppose  that $r=2$. We distinguish cases (see Figure \ref{fig.r1o2}).
	\begin{itemize}
	
	\item If $s\ge 2^2=4$ then,  by Lemma \ref{prop.SW}, $\lambda (\overline{G})\le \lambda (G)$.
	
	\item If $s=2$,  then $G$ is either $P_4$ and $\lambda (\overline{P_4})= \lambda (P_4)=2$, or $G$ is $C_4$ and $\lambda (\overline{C_4})= \lambda (C_4)=2$.

	\item If $s=3$, then $G$ is either $P_5$, $K_{2,3}$,  $K_2(2,3)$,  or the banner $\mathrm{P}$, and
	$\lambda (\overline{P_5})= \lambda (P_5)=2$,
	$\lambda (\overline{K_{2,3}})= \lambda (K_{2,3})=3$,
	$2=\lambda (\overline{K_2(2,3)})< \lambda (K_2(2,3))=3$, and
	$2=\lambda (\overline{\mathrm{P}})<\lambda (\mathrm{P})=3$.
	
	\end{itemize}
\vspace{-.85cm}\end{proof}
%%%%%%%%%%%%%%%%%%%%%%%%%%%%%%%%%%%%%%%%%%%%%%%%%%%%%

	%%%%%%%%%%%%%%%%%%%%%%%%
	%%%%%%%%%%%%%%%%%%%%%%%%
	\begin{figure}[!hbt]
		\begin{center}
			\includegraphics[width=0.65\textwidth]{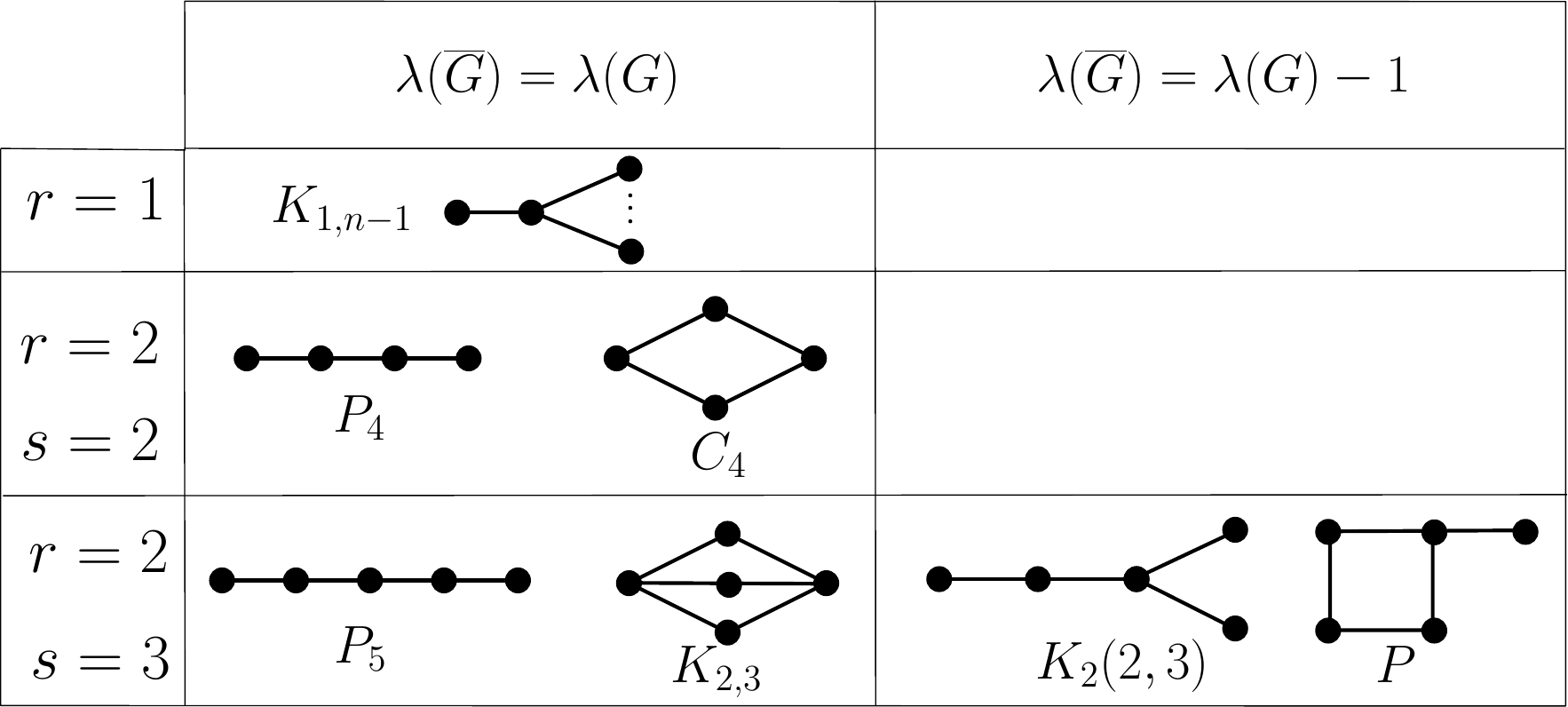}
			\caption{Some bipartite graphs with $1\le r \le 2$.}\label{fig.r1o2}
		\end{center}
	\end{figure}
	%%%%%%%%%%%%%%%%%%%%%%%%
	%%%%%%%%%%%%%%%%%%%%%%%%

Since 
$\lambda (K_2)=1$, $\lambda (\overline{K_2})=2$, and $\lambda (\overline{P_3})=\lambda (\overline{P_3})=2$, by  Proposition~\ref{pro.r1o2} we have that $K_2$ is the only bipartite graph $G$ satisfying $\lambda (\overline{G})=\lambda (G)+1$, whenever $r\in \{ 1,2\}$.
From now on, we assume that $r\ge 3$.

%%%%%%%%%%%%%%%%%%%%%%%%%%%%%%%
%%%%%%%%%%%%%%%%%%%%%%%%%%%%%%%
\begin{proposition}\label{pro.rsiguals}
	If $r=s$, then $\lambda (\overline{G}) \le \lambda (G)$.  
\end{proposition}
\begin{proof}
	If $G$ has an LD-code with vertices at both stable sets, then $\lambda (\overline{G}) \le \lambda (G)$
	by Lemma~\ref{prop.SW}.
	In any other case, $G$ has at most two LD-codes, $U$ and $W$.
	
	If both $U$ and $W$ are LD-codes, then we distinguish the following cases.
	
	\begin{itemize}
	
	\item If there is no vertex $u\in U$ such that $N(u)=W$, then $W$ is an LD-set of $\overline{G}$, and consequently, $\lambda (\overline{G}) \le \lambda (G)$.
	
	\item Analogously, if there is no vertex $w\in W$ such that $N(w)=U$, then we derive $\lambda (\overline{G}) \le \lambda (G)$.
	
	\item If there exist vertices $u\in U$ and $w\in W$ such that $N(u)=W$ and $N(w)=U$, then $(U-\{u\})\cup \{w \}$  would be an LD-set of $\overline{G}$, and thus $\lambda (\overline{G}) \le \lambda (G)$.
	
	\end{itemize}

	Next,  assume that $U$ is an LD-code and $W$ is not an LD-code of $G$. If there is no vertex $w\in W$ such that $N(w)=U$, then $U$ is an LD-set of  $\overline{G}$, and so $\lambda (\overline{G}) \le \lambda (G)$. 
	Finally, suppose that there is a vertex $w\in W$ such that $N(w)=U$. 
	Note that $W$ is not a distinguishing set of $G$ (otherwise, it would be an LD-code because $W$ is a dominating set of size $r$). Therefore, there exist vertices $x,y\in U$ such that $N(x)=N(y)$. In such a case, $(U\setminus \{ x\})\cup \{ w \}$ is an LD-set of $\overline{G}$, and thus $\lambda (\overline{G}) \le \lambda (G)$.
\end{proof}
%%%%%%%%%%%%%%%%%%%%%%%%%%%%%%%
%%%%%%%%%%%%%%%%%%%%%%%%%%%%%%%

From Lemma~\ref{prop.SW} and  Propositions~\ref{pro.r1o2} and \ref{pro.rsiguals} we derive the following result.

%%%%%%%%%%%%%%%%%%%%%%%%%%%%%%%%%%%%%%%%%%%%
%%%%%%%%%%%%%%%%%%%%%%%%%%%%%%%%%%%%%%%%%%%%
\begin{corollary}\label{corin}
	If $\lambda (\overline{G}) =\lambda (G)  + 1$, then $3\le r<s \le 2^r-1$ and $U$ is the only LD-code of $G$.
\end{corollary}
%%%%%%%%%%%%%%%%%%%%%%%%%%%%%%%%%%%%%%%%%%%%
%%%%%%%%%%%%%%%%%%%%%%%%%%%%%%%%%%%%%%%%%%%%

Next theorem characterizes connected bipartite graphs satisfying the equation $\lambda (\overline{G}) =\lambda (G)  + 1$ in terms of the graph associated with a distinguishing set.

%%%%%%%%%%%%%%%%%%%%%%%%%%%%%%%%%%%%%%%%%%%%
%%%%%%%%%%%%%%%%%%%%%%%%%%%%%%%%%%%%%%%%%%%%
\begin{theorem}\label{pro.condnecsuf} Let $3\le r<s$. Then,
	$\lambda (\overline{G}) =\lambda (G)  + 1$
	if and only if the following conditions hold:
	\begin{enumerate}[i)]
		\item $W$ has no twins.
		\item There exists a vertex $w\in W$ such that $N(w)=U$.
		\item For every vertex $u\in U$, the graph $G^U$ has at least two edges with label $u$.
	\end{enumerate}
\end{theorem}
\begin{proof}
	$\Leftarrow )$ 
	Condition $i)$ implies that $U$ is a distinguishing set. Moreover, $U$ is an LD-set of $G$, because $G$ is connected. Hence, $\lambda (G)\le r$.  
	Let $S$ be an LD-code of $\overline{G}$. 
	We next prove that $S$ has at least $r+1$ vertices.
	Condition $ii)$ implies that $U$ is not a dominating set in $\overline{G}$, thus $S\not= U$. 
	If $U\subseteq S$, then $|S|\ge |U|+1=r+1$ and we are done.
	%Let $S_U=S\cap U $.
	If $U\setminus S\not= \emptyset$, consider the graph $G^U$ associated with $U$.
	Let $H_{U\setminus S}$ be the subgraph of  $G^{U}$  induced by the set of edges with a label in $U\setminus S\not= \emptyset$.
	By Corollary~\ref{pro:suprimirVarios},
	the vertices of a same connected component in $H_{U\setminus S}$ have the same neighborhood in $U\cap S$.
	Besides, $W$ induces a complete graph in  $\overline {G}$.
	Hence, $S\cap W$ must contain at least all but one vertex from every connected component of $H_{U\setminus S}$, 
	otherwise $\overline{G}$  would contain vertices with the same neighborhood in $S$.
	Therefore, $S\cap W$ has at least $|V(H_{U\setminus S})|-\cc       (H_{U\setminus S})$ vertices.
	%, where $\cc       (H_{U\setminus S})$ is the number of connected components of  $H_{U\setminus S}$.
		
	Condition iii) implies that there are at least two edges with label $u$, for every $u\in U\setminus S$. Let $H$ be a subgraph of $H_{U\setminus S}$ induced by a set containing exactly two edges with label $u$ for every $u\in U\setminus S$.
	Since $U\setminus S\not= \emptyset$,  the subgraph $H$ has at least two edges.
By Proposition~\ref{claims123} (3), edges with the same label in $G^U$ have no common endpoint, thus we have $|V(H)|\ge 4$. Hence, 
by applying Lemmas~\ref{lem:claim.subgraphcactus} and \ref{lemacactus} we derive
$$|V(H)|-\cc       (H)\ge \frac34 \,  |E(H)|=\frac32 \, |U\setminus S|.$$
Since $H$ is a subgraph of $H_{U\setminus S}$, Lemma~\ref{componentsSubgraf} applies. Therefore 
\begin{align*} |S|&= |S\cap U|+|S\cap W| \\
&\ge  |S\cap U|+|V(H_{U\setminus S})|-\cc       (H_{U\setminus S}) \\
&\ge  |S\cap U|+|V(H)|-\cc       (H) \\
&\ge  |S\cap U|+ \frac 32  |U\setminus S| \\
&= |U| + \frac 12  |U\setminus S| > |U| =r\, .
\end{align*}

	$\Rightarrow)$ By Corollary~\ref{corin},  $U$ is the only LD-code of $G$ and hence,  $U$ is not an LD-set of $\overline{G}$.
	Therefore, $W$ has no twins and $N(w_0)=U$ for some $w_0\in W$.  It only remains to prove that condition iii) holds.
	Suppose on the contrary that there is at most one edge in $G^U$ with label $u$ for some $u\in U$.
	We consider two cases. 
	
	If there is no edge with label $u$, then by Proposition~\ref{pro:basic},  $U\setminus \{ u \}$ distinguishes all pairs of vertices of $W$ in $\overline{G}$.
	Let  $S=(U\setminus \{ u \})\cup \{ w_0 \}$. 
	We claim that $S$ is an LD-set of $\overline{G}$.
	Indeed, $S$ is a dominating set in $\overline{G}$, because in this graph $u$ is adjacent to any vertex in $U\setminus \{ u \}$ and vertices in $W\setminus \{ w_0 \}$ are adjacent to $w_0$. 
	It only remains to prove that $S$ distinguishes the pairs of vertices of the form $u$ and $v$, when $v\in W\setminus \{w_0 \}$. But $w_0\in N_{\overline{G}}(v)\cap S$ and $w_0\notin N_{\overline{G}}(u)\cap S$.
	Thus, $S$ is an LD-set of $\overline{G}$, implying that $\lambda (\overline{G})\le |S|=|U|=\lambda (G)$, a contradiction.
	
	If there is exactly one edge $xy$ with label $u$, then  only one of the vertices $x$ or $y$ is adjacent to $u$ in $G$. Assume that $ux\in E(G)$.  Recall that $x,y \in W$.
	By Proposition~\ref{pro:basic},  $U\setminus \{ u \}$ distinguishes all pairs of vertices of $W$, except the pair $x$ and $y$,  in $\overline{G}$.
	Let  $S=(U\setminus \{ u \})\cup \{ x \}$. 
	We claim that $S$ is an LD-set of $\overline{G}$.
	Indeed, 
	$S$ is a dominating set in $\overline{G}$, because $u$ is adjacent to any vertex in $U\setminus \{ u \}$ and vertices in $W\setminus \{ x \}$ are adjacent to $x$. 
	It only remains to prove that $S$ distinguishes the pairs of vertices of the form $u$ and $v$, when $v\in W\setminus \{x \}$. But $x\in N_{\overline{G}}(v)\cap S$ and $x\notin N_{\overline{G}}(u)\cap S$.
	Thus, $S$ is an LD-set of $\overline{G}$, implying that $\lambda (\overline{G})\le |S|=|U|=\lambda (G)$, a contradiction.
\end{proof}
%%%%%%%%%%%%%%%%%%%%%%%%%%%%%%%%%%%%%%%%%%%%
%%%%%%%%%%%%%%%%%%%%%%%%%%%%%%%%%%%%%%%%%%%%

Observe that condition $iii)$ of Theorem~\ref{pro.condnecsuf}  is equivalent to the existence of at least two pairs of twins in  $G-u$, for every  vertex $u\in U$.
Therefore, it can be stated as follows.

%%%%%%%%%%%%%%%%%%%%%%%%%%%%%%%%%%%%%%%%%%%%
%%%%%%%%%%%%%%%%%%%%%%%%%%%%%%%%%%%%%%%%%%%%
\begin{theorem}  Let $3\le r<s$. Then, 
	$\lambda (\overline{G}) =\lambda (G)  + 1$
	if and only if the following conditions hold:
	\begin{enumerate}[i)]
		\item $W$ has no twins.
		\item There exists a vertex $w\in W$ such that $N(w)=U$.
		\item For every vertex $u\in U$, the graph $G-u$ has at least two pairs of twins in  $W$.
	\end{enumerate}
\end{theorem}
%%%%%%%%%%%%%%%%%%%%%%%%%%%%%%%%%%%%%%%%%%%%
%%%%%%%%%%%%%%%%%%%%%%%%%%%%%%%%%%%%%%%%%%%%

%\merce{El teorema anterior no \'es equivalent a que $U$ sigui l'\'unic LD-code de $G$. Tenim un exemple de graf bipartit amb $U$ \'unic LD-code per\`o que no satisf\`a l'equaci\'o que volem resoldre. Tenim tamb\'e la caracteritzaci\'o de quan $U$ \'es l'\'unic LD-code d'un graf bipartit, per\`o crec que no val la pena posar-ho. En tot cas, si ho volguessi¡m incloure, est\`a mig redactat i comentat al fitxer LaTeX.}

We already know that  it is not possible to have $\lambda (\overline{G})=\lambda (G)+1$ when $s\ge 2^r$. 
However, the condition $s\le 2^r-1$ is not sufficient to ensure the existence of bipartite graphs satisfying $\lambda (\overline{G})=\lambda (G)+1$. 
We next show that there are graphs satisfying this equation if and only if $\frac{3r}2 +1\le s\le 2^r-1$.

\begin{proposition}\label{prop.complementariguanya}
	%Let $G$ be a bipartite graph  such that $ r \ge 3$.
	If $ r \ge 3$ and $\lambda (\overline{G})=\lambda (G)+1$, then $\frac{3r}2+1\le s\le 2^r-1$.
\end{proposition}
%%%%%%%%%%%%%%%%%%%%%%%%%%%%%%

\begin{proof}
	If $ r \ge 3$ and $\lambda (\overline{G})=\lambda (G)+1$, then by Corollary \ref{corin}, we have that $s\le 2^r-1$ and $U$ is the only LD-code of $G$. 
	Moreover, since $G$ satisfies Condition iii) of Theorem~\ref{pro.condnecsuf}, the $U$-associated graph
	$G^U$ contains a subgraph $H$ with exactly two edges labeled with $u$, for every $u\in U$.
	Recall that $V(H)\subseteq V(G)\setminus U=W$.
	Hence, by Corollary \ref{lemacactus}, we have $s=|W|\ge |V(H)|\ge \frac{3r}2+1$.
\end{proof}

%%%%%%%%%%%%%%%%%%%%%%%%%%%%%%%%%%%%%%%%%%%%%%%%%%%%%%%%%%%%%
\begin{proposition}\label{prop.construccio}
	For every pair of integers $r$ and $s$ such that $3\le r$ and $\frac{3r}2 +1\le s\le 2^r-1$, there exists a bipartite graph $G(r,s)$ such that $\lambda (\overline{G})=\lambda (G)+1$.
\end{proposition}
\begin{proof} 
	Let $s=\Big\lceil \frac{3r}2 +1 \Big\rceil$. Let $[r]=\{ 1,2,\dots ,r\}$ and let $\mathcal{P}([r])$ denote the power set of $[r]$. 
	Take the bipartite graph $G(r,\Big\lceil \frac{3r}2 +1 \Big\rceil)$ such that $V=U\cup W$, $U=[r]$, and
	$W\subseteq \mathcal{P}([r])\setminus \{ \emptyset \}$ is defined as follows (see Figure \ref{fig.exemples}).
	For $r=2k$ even:
	\begin{align*}
	W = & \Big\{ [r] \Big\}
	\cup \Big\{ [r]\setminus \{ i \} : i\in [r] \Big\}
	\cup \Big\{ [r]\setminus \{ 2i-1,2i\} : 1\le i\le k \Big\}\phantom{ \cup  \Big\{  [r]\setminus \{ r-1,r \} \Big\}}
	%\phantom{xxxxxxxxxi}
	\end{align*}
	and for $r=2k+1$  odd:
	\begin{align*}
	W = & \Big\{ [r] \Big\}
	\cup \Big\{ [r]\setminus \{ i \} : i\in [r] \Big\}
	\cup \Big\{ [r]\setminus \{ 2i-1,2i\} : 1\le i\le k \Big\}
	%\phantom{
	%xxxxxxxxxxxi}\\
	\cup  \Big\{  [r]\setminus \{ r-1,r \} \Big\} \, .
	\end{align*}		
	
	%%%%%%%%%%%%%%%%%%%%%
	%%%%%%%%%%%%%%%%%%%%%
	\begin{figure}[!hbt]
		\begin{center}
			\includegraphics[height=3.1cm]{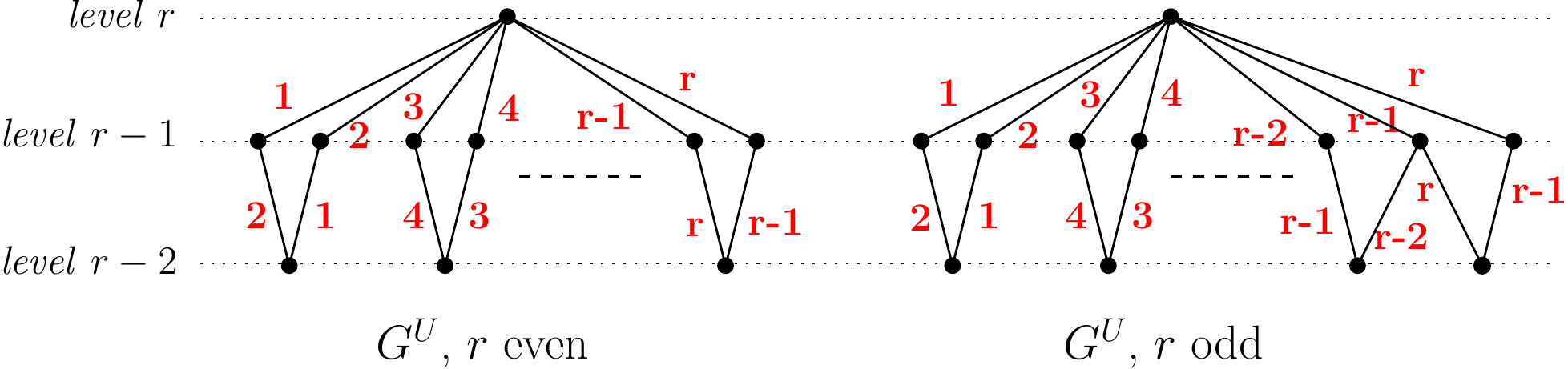}\label{GU}
			\caption{The labeled graph $G^U$, for $G=G(r,\Big\lceil \frac{3r}2 +1 \Big\rceil)$ and $U=\{1, \dots , r \}$.}\label{fig.exemples}
		\end{center}
	\end{figure}	
	%%%%%%%%%%%%%%%%%%%%%
	%%%%%%%%%%%%%%%%%%%%%	
	
The edges of $G(r,\Big\lceil \frac{3r}2 +1 \Big\rceil)$ are defined as follows. 
If $u\in U=[r]$ and $w\in W$, then $u$ and $w$ are adjacent if and only if $u\in w$. 
	
	By construction, $W$ has no twins, there is a vertex $w$ such that $N(w)=U$ and the $U$-associated graph, $G^U$, has at least two edges with label $u$ for every $u\in U$.
	Hence, $\lambda (\overline{G})=\lambda (G) +1$ by Theorem~\ref{pro.condnecsuf}.

%%%
Finally, if $\lceil \frac{3r}2 +1 \rceil<s\le 2^r-1=|\mathcal{P}([r])\setminus \{ \emptyset \} |$, 
consider a set $W'$ obtained by adding $s- \lceil \frac{3r}2 +1 \rceil$ different subsets from $\mathcal{P}([r]) \setminus (W\cup \{ \emptyset \})$ to the set $W$.
Take the bipartite graph $G'$ having $U\cup W'$ as set of vertices and edges defined as before, i.e., for every $u\in U$ and $w\in W'$,
$uw\in E(G')$ if and only if $u\in w$.
Then, $|W'|=s$ and, by construction, $W'$ has no twins.
Moreover, the vertex $[r]\in W\subseteq W' $ satisfies $N([r])=U$ and,  since $W\subseteq W'$,  the U-associated graph $(G')^U$
has at least two edges with label $u$ for every $u\in U$. By Theorem~\ref{pro.condnecsuf}, $\lambda (\overline{G'})=\lambda (G')+1$, and the proof is completed.
%%%
\end{proof}
%%%%%%%%%%%%%%%%%%%%%%%%%%%%%%%%%%%%%%%%%%%%%%%%%%%%%%%%%%%%%%%%%%%%%%

\vspace{-.4cm}
%%%%%%%%%%%%%%%%%%%%%%%%%%%%%%%%%%%%%%%%%%%%%%%%%%%%%%%
%%%%%%%%%%%%%%%%%%%%%%%%%%%%%%%%%%%%%%%%%%%%%%%%%%%%%%%

%%%%%%%%%%%%%%%%%%%%%%%%%%%%%%%%%%%%%%%%%%%%%%%%%%%%%%%
%%%%%%%%%%%%%%%%%%%%%%%%%%%%%%%%%%%%%%%%%%%%%%%%%%%%%%%
%%%%%%%%%%%%%%%%%%%%%%%%%%%%%%%%%%%%%%%%%%%%%%%%%%%%%%%
%%%%%%%%%%%%%%%%%%%%%%%%%%%%%%%%%%%%%%%%%%%%%%%%%%%%%%%

\end{document}